\RequirePackage{tikz}
\documentclass[sn-mathphys]{sn-jnl-1}% Math and Physical Sciences Reference Style
\usepackage{verbatim}
\usepackage{mathtools}

\newcommand{\La}{\mathcal{L}}
\newcommand{\Zeta}{\mathcal{Z}}

\newcommand{\E}{\mathcal{E}}
\newcommand{\V}{\mathcal{V}}
\newcommand{\cL}{\mathbf{L}}
\newcommand{\De}{\mathbf{\Delta}}
\newcommand{\mR}{\mathbf{R}}

\newcommand{\ud}{\mathrm{d}}
\renewcommand\Re{\operatorname{Re}}
\newcommand{\G}{\mathit{\Gamma}}

\jyear{2021}%

\theoremstyle{thmstyleone}%
\newtheorem{theorem}{Theorem}%  meant for continuous numbers
\newtheorem{corollary}{Corollary}% 
\newtheorem{lemma}{Lemma}
\newtheorem{proposition}{Proposition}% to get separate numbers for theorem and proposition etc.

\theoremstyle{thmstyletwo}%

\theoremstyle{thmstylethree}%

\begin{document}

\title[Can One Hear the Spanning Trees of a Quantum Graph?]{Can One Hear the Spanning Trees of a Quantum Graph?}


\author[1]{\fnm{Jonathan} \sur{Harrison}}\email{jon\_harrison@baylor.edu}

\author*[2]{\fnm{Tracy} \sur{Weyand}}\email{weyand@rose-hulman.edu}



\affil[1]{\orgdiv{Department of Mathematics}, \orgname{Baylor University}, \orgaddress{\street{1410 S. 4th Street}, \city{Waco}, \postcode{76706}, \state{TX}, \country{USA}}} % Orchid 0000-0002-2590-4503}}}

\affil*[2]{\orgdiv{Department of Mathematics}, \orgname{Rose-Hulman Institute of Technology}, \orgaddress{\street{5500 Wabash Avenue}, \city{Terre Haute}, \postcode{47803}, \state{IN}, \country{USA}}}  %Orchid 0000-0001-9735-5825}}}

%\affil[3]{\orgdiv{Department}, \orgname{Organization}, \orgaddress{\street{Street}, \city{City}, \postcode{610101}, \state{State}, \country{Country}}}


\abstract{Kirchhoff showed that the number of spanning trees of a graph is the spectral determinant of the combinatorial Laplacian divided by the number of vertices; we reframe this result in the quantum graph setting.   We prove that the spectral determinant of the Laplace operator on a finite connected metric graph with standard (Neummann-Kirchhoff) vertex conditions determines the number of spanning trees when the lengths of the edges of the metric graph are sufficiently close together.  To obtain this result, we analyze an equilateral quantum graph whose spectrum is closely related to spectra of discrete graph operators and then use the continuity of the spectral determinant under perturbations of the edge lengths.}

% 70

%	The abstract serves both as a general introduction to the topic and as a brief, non-technical summary of the main results and their implications. Authors are advised to check the author instructions for the journal they are submitting to for word limits \textbf{150-250 words total} and if structural elements like subheadings, citations, or equations are permitted.}



%\pacs[MSC Classification]{35A01, 65L10, 65L12, 65L20, 65L70}


\maketitle

\section{Introduction}\label{sec1}

The question ``Can one hear the shape of a quantum graph?'' was answered in the affirmative by Gutkin and Smilansky \cite{GS01} for a quantum graph where the set of edge lengths are incommensurate.  Their work was inspired by the famous question of Kac \cite{Kac66} which also sowed the seeds of results on isospectral billiards \cite{Cha95, GWW92}, graphs \cite{BPB09} and Riemannian manifolds \cite{Sun85}, along with many other related works.  The connection between these questions is the extent to which properties of the spectrum can be used to recover information on the system's geometry. 

Isospectral domains were also studied in the context of discrete graphs \cite{Bro99}.  In graph theory, a historic spectral geometric connection is provided by a theorem of Kirchhoff \cite{Kirchhoff}.  Kirchhoff's matrix tree theorem expresses the number of spanning trees of a connected graph in terms of the spectral determinant of the combinatorial Laplacian of the graph which is the matrix $\cL =\mathbf{D}-\mathbf{A}$ where $\mathbf{D}$ is a diagonal matrix of the vertex degrees and $\mathbf{A}$ is the adjacency matrix, see section \ref{sec:background}.

\begin{theorem}[Kirchhoff's Matrix Tree Theorem]\label{lem:matrix tree theorem}
	For a connected graph $G$ with $V$ vertices,
	\begin{equation}
	\# \mbox{ spanning trees } = \dfrac{1}{V}{\det}'(\cL) = \det(\cL[i])
	\end{equation}
	for any $i = 1, 2, \ldots, V$ where $\cL[i]$ is the matrix $\cL$ with row $i$ and column $i$ removed.
\end{theorem}

In this article we reframe Kirchhoff's theorem in the setting of quantum graphs.  Quantum graphs were introduced to model electrons in organic molecules \cite{Pau36} and are widely employed in mathematical physics as a model of quantum mechanics in systems with complex geometry, see \cite{BKbook, Kuc03} for a review. 
In a quantum graph the edges of the graph consist of intervals connected at the vertices with a self-adjoint differential operator on the set of intervals.  Here we consider the most widely studied case of the Laplace operator with Neumann-Kirchhoff (or standard) vertex conditions, denoted $\La$.  For Neumann-Kirchhoff conditions, functions on the graph are continuous and outgoing derivatives at the  vertices sum to zero.  We obtain the following relationship between the spectral determinant of $\La$ and the number of spanning trees of the graph.

\begin{theorem}\label{thm:generic spanning trees}
	Let $\G$ be a connected metric graph with edge lengths in the interval $[\ell,\ell+\delta)$  and $\La$ the Laplacian with Neumann-Kirchoff vertex conditions.  If
	\begin{equation}
	\delta< \frac{\ell}{V^V \, 2^{E+V}\sqrt{2EV}}  \ ,
	\end{equation}
	then the number of spanning trees is the closest integer to
	\begin{equation}\label{eq:T thm}
	T_{\G}=  \dfrac{\prod_{v\in\V} d_v}{E \, 2^E \,\ell^{\beta+1}} {\det}'(\La) \ ,
	\end{equation}
	where $\V$ is the set of vertices, $d_v$ is the degree of vertex $v$, $E$ is the number of edges and $\beta$ is the first Betti number of $\G$.
\end{theorem}

If we have an equilateral quantum graph $\widetilde{\G}$ where all the edge lengths are $\ell$, then $T_{\widetilde{\G}}$ is precisely the number of spanning trees. 
We use $ \widetilde{\La}$ to denote the Laplace operator of an equilateral quantum graph. 

\begin{theorem}\label{thm:equilateral spanning trees}
	Let $ \widetilde{\G}$ be a connected equilateral metric graph with edge length $\ell$ and $\widetilde{\La}$ the Laplacian with Neumann-Kirchoff vertex conditions.  Then,
	\begin{equation}\label{eq:equilateral spanning trees}
	\# \mbox{ spanning trees } =\dfrac{\prod_{v\in\V} d_v}{E \, 2^E \,\ell^{\beta+1}} {\det}'( \widetilde{\La}),
	\end{equation}
	where $\V$ is the set of vertices, $d_v$ is the degree of vertex $v$, $E$ is the number of edges and $\beta$ is the first Betti number of $\widetilde{\G}$.
\end{theorem}

In the formula for $T_\Gamma$ (\ref{eq:T thm}), the number of edges $E$ can be determined from the spectrum of $\La$ via the Weyl law, where the mean spacing of the eigenvalues is given by $\pi^{-1}\sum_{e\in \mathcal{E}} l_e$, see \cite{BKbook}, as the edge lengths are tightly constrained to $[\ell,\ell+\delta)$.  For a connected graph $\Gamma$, the first Betti number is determined by the number of edges and vertices, $\beta=E-V+1$.  Hence, it would be interesting to know if there are also spectral interpretations of the number of vertices and the product of the degrees in (\ref{eq:T thm}).   

The article is organised as follows.  In section \ref{sec:background} we introduce discrete graph and metric graph notation and operators. We relate the spectral determinants of equilateral quantum graphs and discrete graph operators in section \ref{sec:spec det equilateral}.  In section \ref{sec:eqtrees} we use Kirchoff's matrix tree theorem and the relationship between spectral determinants to prove theorem  \ref{thm:equilateral spanning trees}.  In section \ref{sec:det comparison} we compare spectral determinants of equilateral and non-equilateral quantum graphs and in section \ref{sec:general} we prove theorem \ref{thm:generic spanning trees}.  Finally we summarize the results in section \ref{sec:conclusion}.

%We will show, in section \ref{sec:general}, that equation (\ref{eq:equilateral spanning trees}) also evaluates the number of spanning trees when $\La$ is the Laplacian of a metric graph that is not equilateral provided the edge lengths lie in an interval $(\ell, \ell+\delta)$ with $\delta$ sufficiently small.

\section{Background}\label{sec:background}

A \textit{discrete graph} $G$ is comprised of a set of vertices $\mathcal{V}$ and a set of edges $\mathcal{E}$ such that each edge connects a pair of vertices; see for example Figure \ref{fig: complete bipartite}.  Two vertices $u,v\in  \mathcal{V}$ are adjacent if $(u,v)\in \mathcal{E}$ and we write $u\sim v$.  
We will denote the number of vertices by $V = \lvert\mathcal{V}\rvert$ and the number of edges by $E = \lvert\mathcal{E}\rvert$. 
The \textit{degree} of a vertex $v$, denoted $d_v$,  is the number of vertices that are adjacent to $v$.  
We let $\mathcal{E}_v$ denote the set of edges containing the vertex $v$ so $d_v= \lvert \mathcal{E}_v \rvert$.  If $d_v=d$ for all $v\in \mathcal{V}$, the graph is \emph{regular}.
In this paper, we only consider simple graphs where there are no loops or multiple edges and the number of edges and vertices are finite.   
We also assume that $G$ is \textit{connected}, so there is a path between every pair of vertices. 
The \emph{distance} between a pair of vertices $u,v\in \mathcal{V}$ on a connected discrete graph is the number of edges in a shortest path joining $u$ and $v$.  
The \emph{diameter} $D$ of the graph is the maximum distance between a pair of vertices.
The \textit{first Betti number} of $G$ is $\beta = E-V+1$, the number of independent cycles on $G$.  A \textit{tree} is a connected graph with no cycles, $\beta = 0$. A \textit{spanning tree} of $G$ is a subgraph of $G$ that is a tree and contains all the vertices of $G$.  

\begin{figure}[tbh]
\begin{center}
\begin{tikzpicture}
  [scale=1.5,every node/.style={circle,fill=black!, scale=.5}]
  \node (n1) at (0,.5) {};
  \node (n2) at (0,-.5)  {};
  \node (n3) at (1,0.75)  {};
  \node (n4) at (1,0.25) {};
  \node (n5) at (1,-0.25) {};
  \node (n6) at (1, -0.75) {};

  \foreach \from/\to in {n1/n3,n1/n4,n1/n5, n1/n6, n2/n3,n2/n4,n2/n5, n2/n6}
    \draw (\from) -- (\to);

\end{tikzpicture}
\end{center}
\caption{\small{The complete bipartite graph $K_{2,4}$.}}\label{fig: complete bipartite}
\end{figure}
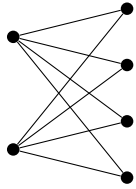 

A function on a discrete graph $G$ takes values at the vertices and therefore can be viewed as a vector $f\in \mathbb{C}^V$.   The \textit{combinatorial Laplace operator}, $\cL$, is defined by
\begin{equation}
(\cL f)(v) = d_v f(v) - \sum_{u\sim v} f(u).
\end{equation}
 This means that we can define the self-adjoint $V\times V$ matrix $\cL$ where
\begin{equation} \cL_{uv} = \begin{cases}
	d_u & \text{if } u=v\\
	-1 & \text{if } u\sim  v \\
	0 & \text{otherwise} 
	\end{cases} \ .
\end{equation}
We note that $\cL$ has zero as an eigenvalue (with multiplicity one for a connected graph), and we will index the eigenvalues of $\cL$ in increasing order,
\begin{equation}
0 = \mu_1 < \mu_2 \leq \mu_3 \leq \ldots \leq \mu_V.
\end{equation}
The harmonic Laplacian is 
\begin{equation}
\De = \mathbf{D}^{-1} \cL 
\end{equation}
where $\mathbf{D}$ is a diagonal matrix of the degrees, $\mathbf{D}_{uv} =d_v \delta_{uv}$.  We denote the eigenvalues of $\De$ with $\lambda_j$ for $j=1,\dots, V$.

A \textit{metric graph} $\G$ is a discrete graph where every edge $e\in \mathcal{E}$ is assigned a length $\ell_e>0$.   We associate the edge $e$ with the interval $[0,\ell_e]$. For $e=(u,v)$, we set $x_e=0$ at $u$ and $x_e=\ell_e$ at $v$; the choice of orientation of the coordinate is arbitrary and will not affect the results.   The \textit{total length} of a metric graph is the sum of the lengths of each edge,  $ \ell_{\mathrm{tot}} = \sum_{e\in\E} \ell_e$.    An \textit{equilateral metric graph} $ \widetilde{\G}$  is a metric graph where all the edge lengths are equal.  We will use the term \emph{generic metric graph} to distinguish a non-equilateral graph.  
%\textcolor{red}{Given a metric graph $\G$, there is clearly a corresponding discrete graph}.  
A function $f$ on $\G$ is a collection of functions $\{f_e\}_{e\in\mathcal{E}}$  where $f_e$ is a function on the interval $[0,\ell_e]$.

A \textit{quantum graph} is a metric graph $\G$ along with a self-adjoint differential operator. In this paper, we consider the \textit{Laplace operator}, denoted by $\La$, which acts as $-\frac{\ud^2}{\ud x_e^2}$ on functions defined on $[0,\ell_e]$. We will use the notation $\widetilde{\La}$ if $ \widetilde{\G}$ is an equilateral metric graph. At the vertices, functions will satisfy the Neumann-Kirchhoff (or standard) vertex conditions,
\begin{equation}\label{eq: BCs}
\left\{\begin{array}{l}
f(x) \mbox{ is continuous on } \G \mbox{ and }\\
\displaystyle\sum_{e \in \mathcal{E}_v} \dfrac{\ud f_e}{\ud x_e}(v)=0 \mbox{ at each vertex } v
\end{array}\right. \ .
\end{equation}
By convention, $\dfrac{\ud f_e}{\ud x_e}(v)$ is taken to be the outgoing derivative at the end of the interval $[0,\ell_e]$ corresponding to $v$. The second Sobolev space on $\G$ is the direct sum of the second Sobolev spaces on the set of intervals,
\begin{equation}
H^2(\G) = \bigoplus_{e\in\mathcal{E}} H^2([0,\ell_e]).
\end{equation}
The domain of $\La$ is the set of functions  $f\in H^2(\G)$ that satisfy \eqref{eq: BCs}. With these vertex conditions, $\La$ has an infinite number of non-negative eigenvalues \cite{BKbook} which we will denote as
\begin{equation}
0\leq k_1^2 \leq k_2^2 \leq \ldots
\end{equation}
with $k_j \in \mathbb{R}$. 

Some of these eigenvalues may correspond to eigenfunctions where $f_e(0) = f_e(\ell_e)=0$ for all $e\in \mathcal{E}$.  In this case the eigenfunction is also an eigenfunction of the Laplacian with Dirichlet vertex conditions.  The spectrum of the graph with Dirichlet vertex conditions is just the union of the spectra of the Laplacian on $E$ disconnected intervals $[0,\ell_e]$ with Dirichlet boundary conditions.  We call this the \emph{Dirichlet spectrum} of $\G$.

\section{Spectral determinants of equilateral graphs}\label{sec:spec det equilateral}

In this section we relate the spectral determinants of $\De$ and $ \widetilde{\La}$. This is used in section \ref{sec:eqtrees} to prove theorem 
\ref{thm:equilateral spanning trees}. The \textit{spectral determinant} of the harmonic Laplacian $\De$ is the product of its nonzero eigenvalues, 
\begin{equation}
{\det}^\prime(\De)={\prod_{j=1}^V}^\prime \lambda_j \ .
\end{equation}
The prime shows that eigenvalues of zero are omitted from the product. 
Correspondingly, the spectral determinant of the Laplace operator $ \widetilde{\La}$ on an equilateral metric graph is, formally,
\begin{equation}\label{eq:detLa}
{\det}^\prime( \widetilde{\La}) = {\prod_{j=1}^\infty}^\prime k^2_j \,. 
\end{equation}

The spectral determinant of quantum graphs has been studied in a number of situations \cite{Des00, Des01, Fri06, HarKir11, HarKirTex12}.  Here we use a zeta function regularization of the product.
The \textit{spectral zeta function}, $\Zeta(s)$, is a generalization of the Riemann zeta function where the nonzero eigenvalues take the place of the integers. The spectral zeta function corresponding to $ \widetilde{\La}$ is
\begin{equation}
\Zeta(s) = \sum_{j=1}^\infty k_j^{-2s},
\end{equation}
which converges for $\Re (s)> 1$.  Making an analytic continuation to the left of $\Re (s)=1$, the regularized spectral determinant is defined as 
\begin{equation}
{\det}^\prime( \widetilde{\La}) = \mbox{exp}(-\Zeta'(0)).
\end{equation}

There is a correspondence between the eigenvalues of $\De$ and $\widetilde{\La}$ \cite{Below85} (see also \cite{Kuc03, Pank06, BKbook}).
\begin{proposition}\label{lemma: relation}
Suppose $G$ is a discrete graph and $\widetilde{\G}$ is the corresponding equilateral metric graph with edge length $\ell$.  If $ k^2$ is not in the Dirichlet spectrum of $ \widetilde{\La}$, then
\begin{equation}
k^2 \in \sigma( \widetilde{\La}) \iff  1-\cos(k\ell) \in \sigma(\De)
\end{equation}
where $\sigma(\cdot)$ denotes the spectrum of the operator.
\end{proposition}

 By proposition \ref{lemma: relation}, every eigenvalue of $\De$ can be written as $1-\cos(t\ell)$ for some $t \in \left[0,\frac{\pi}{\ell}\right]$. Let $T = \{t_j\}_{j=1}^V$ be the set $0 = t_1 < t_2 \leq \ldots \leq t_V \leq \frac{\pi}{\ell}$ such that $1-\cos(t_j\ell) \in \sigma(\De)$ (including multiplicity).  We know $0=t_1< t_2$ as zero is an eigenvalue of $\De$ with multiplicity one  when $G$ is connected. This allows us to write the spectral zeta function of $ \widetilde{\La}$ in terms of the finite set $T$. The following lemma is a combination of equations (20) and (24) from  \cite{HarWey18}.
\begin{lemma}\label{lemma: zeta function}
Suppose $G$ is a discrete graph and $ \widetilde{\G}$ is the corresponding equilateral metric graph with edge length $\ell$. For $\Re(s) \neq \frac{1}{2}$, the spectral zeta function of $ \widetilde{\La}$ is
\begin{align}\label{eq:zeta function}
\Zeta (s) &= \left(4^s(\beta - 1) + 2\right) \left(\frac{\ell}{2\pi}\right)^{2s}\zeta_R(2s)\nonumber\\\hspace{.5cm}&+\left(\dfrac{\ell}{2\pi}\right)^{2s}\sum_{j=2}^V \left(\zeta_H\left(2s,\frac{t_j\ell}{2\pi}\right) + \zeta_H\left(2s,1-\frac{t_j\ell}{2\pi}\right)\right) 
\end{align}
where $\zeta_R(z)$ is the Riemann zeta function, $\zeta_H(z,a)$ is the Hurwitz zeta function, and $t_j\in T$.
\end{lemma}
 %The Hurwitz zeta functions in \eqref{eq:zeta function} are written with the arguments $t_j\ell / 2\pi$ and $1-t_j\ell / 2\pi$  because $\zeta_H(z,a)$ is analytic for $\Re(a) > 0$.

Using lemma \ref{lemma: zeta function} we can obtain a relationship between the spectral determinants of the harmonic Laplacian of a discrete graph and the Laplacian of the corresponding equilateral quantum graph.
\begin{proposition}\label{theorem: spectral determinant}
Suppose $G$ is a connected discrete graph and $ \widetilde{\G}$ is the corresponding equilateral metric graph with edge length $\ell$. Then 
\begin{equation}
{\det}^\prime( \widetilde{\La}) = 2^{E-1}\ell^{\beta+1}{\det}^\prime(\De)\ .
\end{equation}
%where $\beta = E-V+1$ is the first Betti number of the graph.
\end{proposition}

\begin{proof}%[Proof of proposition \ref{theorem: spectral determinant}]

Since $\zeta_H(0,a) = \frac{1}{2}-a$ \cite{NITS}, it follows that,
\begin{equation}
\zeta_H(0,a) + \zeta_H(0,1-a) =0 \ .
\end{equation}
Using lemma \ref{lemma: zeta function}, 
\begin{equation}
\Zeta '(0) = -\ln\left(2^{\beta-1}\right) - \ln(\ell^{\beta+1})   + 2\sum_{j=2}^V \left[\ln\left(\Gamma\left(\frac{t_j\ell}{2\pi}\right)\Gamma\left(1-\frac{t_j\ell}{2\pi}\right)\right)- \ln{2\pi}\right]\\
\end{equation}
since $\zeta_R(0)=-1/2$, $\zeta'_R(0) = -\ln(2\pi)/2$, and $\zeta'_H(0,a) = \ln(\Gamma(a))-\frac{1}{2}\ln(2\pi)$ \cite{NITS}. Additionally $\Gamma(z)\Gamma(1-z) = \pi / \sin(\pi z)$ \cite{NITS} and $4\sin^2\left(\frac{t_j\ell}{2}\right) = 2(1-\cos(t_j\ell) )= 2\lambda_j$, so 
\begin{align}
-\Zeta '(0)&=  \ln\left(2^{\beta-1}\right) + \ln(\ell^{\beta+1})   + \sum_{j=2}^V \ln\left(4\sin^2\left(\frac{t_j\ell}{2}\right)\right)\\
&=   \ln\left(2^{V + \beta-2}\ell^{\beta+1}\prod_{j=2}^V \lambda_j \right).
\end{align}
Taking the exponential of this expression yields the result.

\end{proof}

\subsection{Example: complete bipartite graphs}

Proposition \ref{theorem: spectral determinant} simplifies for complete bipartite graphs. 
The \textit{complete bipartite graph}, $K_{m,p}$, consists of $m+p$ vertices and $mp$ edges. The vertices can be divided into two disjoint sets, $\mathcal{M}$ of size $m$ and $\mathcal{P}$ of size $p$, such that every vertex in $\mathcal{M}$ is connected to every vertex in $\mathcal{P}$ and no vertex is connected to another vertex from the same set; see for example Figure \ref{fig: complete bipartite}.   A \textit{star graph} is a complete bipartite graph with $m=1$ and $p=V-1 = E$.

\begin{corollary}\label{cor: complete bipartitie}
For an equilateral complete bipartite graph $K_{m,p}$ with edge length $\ell$,
\begin{equation}\label{eq:complete}
{\det}^\prime( \widetilde{\La}) = 2^{mp}\ell^{mp-m-p+2}.
\end{equation}
In particular, for a star graph, 
\begin{equation}\label{eq: star}
{\det}^\prime( \widetilde{\La}) = 2^E\ell.
\end{equation}
\end{corollary}

\begin{proof}
The eigenvalues of a discrete complete bipartite graph $K_{m,p}$ are $0$ with multiplicity one, $1$ with multiplicity $m+p-2$, and $2$ with multiplicity one \cite{Chungbook}, and therefore
\begin{equation}
{\det}^\prime(\De) = 2.
\end{equation}
The complete bipartite graph $K_{m,p}$ has $\beta = E-V+1 = mp-(m+p)+1$. Substituting into  proposition \ref{theorem: spectral determinant} produces the result.  
The case of a star graph is obtained by setting $m=1$ and $p=V-1=E$. 
\end{proof}

%Equation \eqref{eq:complete} agrees with \cite{HarWey18} where the spectral determinant for a complete bipartite graph was calculated via the spectral zeta function. 

Corollary \ref{cor: complete bipartitie} agrees with \cite{HarWey18} where the spectral determinant was obtained from the quantum  spectral zeta function directly. The star graph formula \eqref{eq: star} agrees with \cite{HarKir11} where the spectral zeta functions of quantum star graphs were calculated using a contour integral approach.

\section{Spanning trees of equilateral graphs}\label{sec:eqtrees}

Kirchhoff's matrix tree theorem \cite{Kirchhoff} gives the number of spanning trees of a discrete graph in terms of the spectral determinant of the combinatorial Laplacian. In this section we prove theorem \ref{thm:equilateral spanning trees} which reframes this result in terms of the spectral determinant of an equilateral quantum graph.

\subsection{Regular equilateral graphs}

We start by proving theorem \ref{thm:equilateral spanning trees} for a $d$-regular graph where the proof is straightforward.

\begin{proposition}\label{thm: regular}
Suppose $ \widetilde{\G}$ is a connected $d$-regular equilateral metric graph with edge length $\ell$. Then
\begin{equation}
\# \mbox{ spanning trees } =
\dfrac{d^{V-1}}{2^{E-1}\ell^{\beta+1}V}{\det}'( \widetilde{\La}).
\end{equation}
\end{proposition}

\begin{proof}
Since $ \widetilde{\G}$ is a $d$-regular graph,
\begin{equation}
\De = \mathbf{D}^{-1}\cL = \dfrac{1}{d}\cL.
\end{equation}
If $\mu$ is an eigenvalue of $\cL$, then $\mu / d$ is an eigenvalue of $\De$ so
\begin{equation}
{\det}'(\De) = \prod_{i=2}^V \dfrac{\mu_i}{d}=d^{1-V}{\det}'(\cL).
\end{equation}
By Kirchhoff's matrix tree theorem,
\begin{equation}\label{eq:regular}
\# \mbox{ spanning trees } =\dfrac{d^{V-1}}{V}\mbox{det}'(\De).
\end{equation}
Combining equation \eqref{eq:regular} with proposition \ref{theorem: spectral determinant} produces the result.
\end{proof}

\subsection{General equilateral graphs}

Theorem \ref{thm:equilateral spanning trees} follows from proposition  \ref{theorem: spectral determinant} and the next lemma, which is an adaptation of a theorem by Chung and Yau \cite{ChYa} and generalizes equation \eqref{eq:regular}.

\begin{lemma} \label{thm:discretelemma}
For a connected discrete graph $G$, 
\begin{equation}
\# \mbox{ spanning trees } =\dfrac{\prod_{v\in\V} d_v}{2E}{\det}'(\De).
\end{equation}
%where $d_i$ is the degree of vertex $i$.
\end{lemma}

\begin{proof}
Since $\De$ has one eigenvalue of zero, its characteristic polynomial is
\begin{equation}
p(x) = \det(\De-x\mathbf{I}) = -x\prod_{i=2}^V (\lambda_i-x) \ .
\end{equation}
Therefore the coefficient of the linear term in $p(x)$ is
\begin{equation}\label{eq:coefficient1}
-\prod_{i=2}^V \lambda_i = -{\det}'(\De).
\end{equation}
On the other hand, since $\De = \mathbf{D}^{-1}{\cL}$,
\begin{equation}
\De - x\mathbf{I}  = \mathbf{D}^{-1}(\cL-x\mathbf{D}),
\end{equation}
and the characteristic polynomial can also be written as
\begin{equation}
p(x) = \left(\prod_{v\in\V} d_v\right)^{-1}\det(\cL-x\mathbf{D}).
\end{equation}
Determinants are linear along the rows of a matrix so
\begin{equation}
\det(\mathbf{A}+\mathbf{B}) = \sum_{S \subseteq [n]} \det \mathbf{B}_S
\end{equation}
where $\mathbf{A}$ and $\mathbf{B}$ are square matrices of size $n$, $[n]=\{1,2,\ldots,n\}$, and $\mathbf{B}_S$ is the matrix $\mathbf{B}$ whose rows that are indexed by $S$ are replaced with the corresponding rows from $\mathbf{A}$. Hence,
\begin{align}
\det(\cL-x\mathbf{D}) &= \det(\cL) + \sum_{\substack{S \subset [V]\\
                  \lvert S\rvert = V-1\\}} \det(x\mathbf{D}_S) + \sum_{\substack{S \subset [V]\\ \lvert S\rvert = V-2}}\det(x\mathbf{D}_S) + \ldots + \det(x\mathbf{D})\\%\nonumber
&= \det(\cL) -x\sum_{i=1}^V d_i\det(\cL[i]) + c_2x^2 + \ldots + c_nx^n
\end{align}
where $\cL[i]$ is the matrix $\cL$ with both row $i$ and column $i$ removed.   By theorem \ref{lem:matrix tree theorem}, $\det(\cL[i])$ is the number of spanning trees of $G$. Therefore, we can see that
\begin{equation}\label{eq:coefficient2}
{\det}'(\De) = \dfrac{\sum_{v\in\V} d_v}{\prod_{v\in \V} d_v} \times 
\# \mbox{ spanning trees } \, ,
\end{equation}
which proves the lemma as $\sum_{v\in\V} d_v = 2E$.
\end{proof}

We have now established theorem \ref{thm:equilateral spanning trees}.  So, for example, a star graph has $\beta=0$ and one vertex has degree $E$ while the others have degree $1$.  Then,
using the spectral determinant of an equilateral star graph (corollary \ref{cor: complete bipartitie}), we see
\begin{equation}
\# \mbox{ spanning trees } =\dfrac{\prod_{v\in\V} d_v}{E2^E \ell^{\beta +1}} 2^E \ell = 1
\end{equation}
as required. 

\section{Spectral determinants of generic quantum graphs}\label{sec:det comparison}

In this section we compare the spectral determinants of equilateral and generic quantum graphs where the corresponding discrete graphs are the same.  In particular, given an equilateral graph $ \widetilde{\G}$ with edge length $\ell$, we will consider a generic graph whose edge lengths lie in the interval $[\ell, \ell +\delta)$.   Friedlander  computed the spectral determinant of a generic quantum graph \cite{Fri06}.

\begin{theorem}\label{thm:Friedlander}
	Suppose $\G$ is a connected metric graph. Then
	\begin{equation}\label{eq:Friedlander}
	{\det}'(\La) = \frac{2^E \ell_{\mathrm{tot}}}{V} \frac{\prod_{e\in\E} \ell_e}{\prod_{v\in\V} d_v} {\det}'(\mR) 
	\end{equation}
	where $\mR$ is the $V\times V$ matrix defined by
	\begin{equation}\label{eq: R}
	\mR_{uv}= \begin{cases}
	\sum_{w\sim v} \ell_{(w,v)}^{-1} & \text{if } u=v\\
	-\ell_{(u,v)}^{-1} & \text{if } u\sim  v \\
	0 & \text{otherwise} 
	\end{cases} \ .
	\end{equation}
\end{theorem}

First, we note that the proof of lemma \ref{thm:discretelemma} also showed
\begin{equation}
{\det}'(\cL) =\frac{V\prod_{v\in\V} d_v}{2E} {\det}'(\De) \ .
\end{equation}
In the case of an equilateral graph with edge length $\ell$, we will define $\widetilde{\mR} = \ell^{-1}\cL$, which agrees with \eqref{eq: R} when $\ell_e = \ell$ for every edge $e$.  In fact, setting the edge lengths equal in theorem \ref{thm:Friedlander}, 
\begin{align}
{\det}' (\La) \rvert_{l_e=\ell} & = \frac{2^E E \ell^{E-V+2}}{V \prod_{v\in\V} d_v} {\det}'(\cL) \\
&= 2^{E-1} \ell^{\beta+1} {\det}'(\De) \\
&= {\det}' ( \widetilde{\La})\ ,
\end{align}
where we used proposition \ref{theorem: spectral determinant} for the last step.
This demonstrates that evaluating (\ref{eq:Friedlander}) with equal edge lengths produces the spectral determinant of the equilateral graph as expected.  

%It remains to bound $\lvert {\det}'(\La_g)-{\det}'( \widetilde{\La})\rvert$ for $l_j$ in $(\ell-\delta,\ell+\delta)$. \textcolor{red}{What interval?} 

\subsection{Bound on the spectral norm of $\widetilde{R}-R$}\label{sec:norm}
Let $\ell_e=\ell+\delta_e$ where $\delta_e\in [0,\delta) $.  Then $ \lvert\ell^{-1}- \ell_e^{-1}  \rvert < \delta \ell^{-2} $.  
To bound the spectral norm $\lvert \lvert \widetilde{\mR}-\mR \rvert \rvert_2$, we use the fact that $\lvert \lvert \mathbf{A}\rvert \rvert_2\leq \sqrt{\sum_{ij} \lvert \mathbf{A}_{ij} \rvert^2}$ for a matrix $\mathbf{A}$.  From \eqref{eq: R},

\begin{equation}\label{eq:Friedlander difference}
\widetilde{\mR}_{uv}-\mR_{uv}= \begin{cases}
\sum_{r\sim u} \frac{\delta_{(r,u)}}{(\ell+\delta_{(r,u)})\ell} & \text{if } u=v\\
-\frac{\delta_{(u,v)}}{(\ell+\delta_{(u,v)})\ell} & \text{if } u\sim  v \\
0 & \text{otherwise} 
\end{cases} \ .
\end{equation}
Hence by Bergstr\"om's inequality, 
\begin{align}
\sum_v \lvert \widetilde{\mR}_{uv}-\mR_{uv} \rvert^2 & = \left( \sum_{r\sim u} \frac{\delta_{(r,u)}}{(\ell+\delta_{(r,u)})\ell} \right)^2+ \sum_{u\sim v} \left( \frac{\delta_{(u,v)}}{(\ell+\delta_{(u,v)})\ell}  \right)^2 \\
&< (d_u+1) \sum_{u\sim v} \delta^2 \ell^{-4} \ .
\end{align}
From this we see,
\begin{equation}
\lvert \lvert \widetilde{\mR} -\mR \rvert \rvert_2 < \frac{\delta \sqrt{2E (d_{\textrm{max}}+1)}}{\ell^2}< \frac{\delta \sqrt{2EV}}{\ell^2}\ ,
\end{equation}
where $d_{\textrm{max}}$ is the maximum degree of any vertex.
 
Consequently, if the eigenvalues of $\mR$ are $0=\lambda_1<\lambda_2<\dots<\lambda_{V}$ and the eigenvalues of $\widetilde{\mR}$ are
 $0=\tilde{\lambda}_1<\tilde{\lambda}_2<\dots<\tilde{\lambda}_{V}$ then,
 \begin{equation}\label{eq: normbound}
\lvert \tilde{\lambda}_j - \lambda_j \rvert < \lvert \lvert \widetilde{\mR}-\mR \rvert \rvert_2 < \frac{\delta\sqrt{2EV} }{\ell^2} \ .
 \end{equation}

\subsection{A bound on the change in a spectral determinant}\label{sec:det}

Given a set of real numbers $0<a<\alpha_1\leq \alpha_2 \leq \dots \leq \alpha_n$,
\begin{align}
&\prod_{j=1}^n (\alpha_j+a)-\prod_{j=1}^n \alpha_j\nonumber\\
&= a^n +a^{n-1}\sum_{i=1}^n \alpha_i+a^{n-2}\sum_{i_1\neq i_2} \alpha_{i_1}\alpha_{i_2}+\dots+a\sum_{\substack{i_1,\dots, i_{n-1}:\\ i_j\neq i_k}} \alpha_{i_1}\dots \alpha_{i_{V-1}}\\
&< a \left( \prod_{i=2}^n \alpha_i \right)  \left(\sum_{j=0}^{n-1} {n\choose j} \right)\\ & < a2^n\left( \prod_{i=2}^n \alpha_i \right)\label{eq: detbound} \ .
\end{align}

Now we compare the spectral determinants of $\mR$ and $\widetilde{\mR}$. From \eqref{eq: normbound}, we know that the difference between the corresponding eigenvalues of $\mR$ and $\widetilde{\mR}$ is at most
\begin{equation}
a = \frac{\delta\sqrt{2EV}}{\ell^2}.
\end{equation}
In our situation, we have $0={\lambda}_1<a<{\lambda}_2 \leq {\lambda}_3 \leq \ldots \leq {\lambda}_V$ (assuming that $\delta$ is small enough so that $a < {\lambda}_2$). Then from  \eqref{eq: detbound},
\begin{equation}\label{eq: R-r}
\lvert {\det}'(\mR) - {\det}'(\widetilde{\mR})\rvert < \frac{\delta2^{V-1}\sqrt{2EV}}{\ell^2 {\lambda}_2}{\det}'(\mR) \ .
\end{equation}

\section{Spanning trees of generic quantum graphs}\label{sec:general}

For a generic quantum graph $\G$ with edge lengths in $[\ell,\ell+\delta)$, let  
\begin{equation}\label{eq: T}
T_{\G}=  \dfrac{\prod_{v\in\V} d_v}{E \, 2^E \,\ell^{\beta+1}} {\det}'(\La) \ .
\end{equation}
Notice that, by theorem \ref{thm:equilateral spanning trees}, $T_{\widetilde{\G}}$ is the number of spanning trees of an equilateral graph.   For $\delta$ sufficiently small, the value of $T_{\G}$ will be close enough to $T_{\widetilde{\G}}$ to determine the number of spanning trees of a generic quantum graph. 
Using theorem \ref{thm:Friedlander},
\begin{equation}\label{eq:T with R}
T_{\G} = \frac{\ell_{\mathrm{tot}} \prod_{e\in\E}\ell_e}{EV\ell^{\beta+1}}{\det}'(\mR) \ ,
\end{equation}
and similarly,
\begin{equation}\label{eq:Te}
T_{\widetilde{\G}} = \frac{\ell^{E-\beta}}{V}{\det}'(\widetilde{\mR}) \ .
\end{equation}
 We will determine a bound on $\delta$ such that  $\vert T_{\G} - T_{\widetilde{\G}} \rvert <1/2$, and hence the number of spanning trees is the closest integer to $T_{\G}$. To this end, we employ two bounds on the spectrum of $\cL$.  
The first is a lower bound on the second smallest eigenvalue in terms of the graph diameter $D$, the maximum distance between a pair of vertices, due to McKay \cite{Moh91}.
\begin{theorem}\label{thm:lower bound Moh91}
The second smallest eigenvalue of $\cL$ is bounded below, $\mu_2\geq \dfrac{4}{DV}$.
\end{theorem}
\noindent The second is an upper bound on the eigenvalues of $\cL$ , see \cite{Kel67, AndMor, Mer94}.
\begin{theorem}\label{thm:upper bound on spectrum of L}
	The eigenvalues of $\cL$ are bounded above, $\lambda_j\leq V$ for $j=1,\dots, V$.
\end{theorem}
%Consequently, $\tilde{\lambda}_2 > 4/V^2 \ell$ and ${\det}'(\cL) \leq V!$.

\begin{proof}[Proof of Theorem \ref{thm:generic spanning trees}]
From equations (\ref{eq:T with R}) and (\ref{eq:Te}),
\begin{equation}\label{eq: inequality1}
\lvert 	T_{\G} - 	T_{ \widetilde{\G}}  \rvert
\leq \frac{\ell_{\mathrm{tot}} \prod_{e\in\E} \ell_e}{EV\ell^{\beta+1}} \lvert {\det}'(\mR) - {\det}'(\widetilde{\mR}) \rvert 
+ \frac{\lvert {{\det}' (\widetilde{\mR})} \rvert}{V}   \left\lvert \frac{\ell_{\mathrm{tot}} \prod_{e\in\E} \ell_e}{E\ell^{\beta+1}} - \ell^{E-\beta}  \right\rvert. 
\end{equation}
Using \eqref{eq: R-r} and assuming that $\delta< \ell$,
\begin{equation}
\lvert 	T_{\G} - 	T_{ \widetilde{\G}}  \rvert < \frac{(2\ell)^{E+1}}{V\ell^{\beta+1}} \frac{2^{V-1}\sqrt{2EV} \delta \lvert {\det}' (\widetilde{\mR}) \rvert }{\ell^2\tilde{\lambda}_2} 
+   \frac{\lvert {{\det}' (\widetilde{\mR})} \rvert}{V\ell^{\beta+1}}   \left\lvert (\ell+\delta)^{E+1} - \ell^{E+1}  \right\rvert.
\end{equation}
Using \eqref{eq: detbound}, we can see that
\begin{equation}\label{eq: ldifference}
\lvert (\ell+\delta)^{E+1} - \ell^{E+1}\rvert < \delta2^{E+1}\ell^E \ .
\end{equation}
Consequently,
\begin{equation}
\lvert 	T_{\G} - 	T_{ \widetilde{\G}}  \rvert
  < \delta \lvert {\det}'(\cL) \rvert \frac{2^{E+1}}{\ell V}
\left[ \frac{2^{V-1}\sqrt{2EV}  }{\ell \tilde{\lambda}_2} + 1 \right]
\end{equation}
since ${\det}'(\widetilde{\mR}) = \ell^{-(V-1)}{\det}'(\cL)$.

 We can conclude from theorem \ref{thm:lower bound Moh91} that $\tilde{\lambda}_2 > 4/V^2\ell$, and we know by applying theorem \ref{thm:upper bound on spectrum of L} that ${\det}'(\cL) \leq V^{V-1}$. Using these inequalities we see that
\begin{equation}
\lvert 	T_{\G} - 	T_{ \widetilde{\G}}  \rvert
< \frac{\delta}{\ell} V^{V-2}\, 2^{E+1}
\left[V^2 2^{V-3}\sqrt{2EV}  + 1 \right]<\frac{\delta}{\ell} V^V \, 2^{E+V-1}\sqrt{2EV} ,
\end{equation}
which establishes theorem \ref{thm:generic spanning trees}.
\end{proof}

\subsection{Star graph example}
\label{sec: star}

For comparison, we write an equivalent bound in the case of the star graph where there are explicit formulae for the spectral determinants of the equilateral and generic quantum graphs.  From \cite{HarKir11},
\begin{equation}\label{eq:generic star det}
{\det}'(\La)=\frac{2^E}{E} \sum_{e\in\E} \ell_e \ ,
\end{equation}
which agrees with the spectral determinant of the equilateral graph, corollary \ref{cor: complete bipartitie}.  For $\ell_e \in [\ell,\ell +\delta)$,
\begin{equation}
\lvert {\det}'( \widetilde{\La})-{\det}'(\La)  \rvert < 2^E\delta \ .
\end{equation}
Hence,
\begin{equation}
\lvert T_{\G}-T_{\widetilde{\G}}\rvert < \frac{\delta}{\ell} \ .
\end{equation}
Therefore, the closest integer to $T_{\G}$ is the number of spanning trees if $\delta \leq \ell/2$.

If we have a star graph, then $E=V-1$ and the condition on $\delta$ from theorem \ref{thm:generic spanning trees} is
\begin{equation}
\delta< \frac{\ell}{V^V \, 2^{2V-1}\sqrt{2V(V-1)}} \ .
\end{equation}
Clearly, the demand on the edge lengths in theorem \ref{thm:generic spanning trees} is suboptimal so, in fact, one may expect that the nearest integer to $T_{\G}$ gives the number of spanning trees even when the edge lengths are less tightly constrained.

\section{Discussion}\label{sec:conclusion}
In this paper, we proved an analog of Kirchhoff's matrix tree theorem for quantum graphs. In particular, we determined the number of spanning trees of an equilateral quantum graph from its spectral determinant. To do this we related the spectral determinant of an equilateral quantum graph to the spectral determinant of the harmonic Laplacian of the corresponding discrete graph. We extended this to non-equilateral  quantum graphs where the edge lengths are sufficiently constrained.  

The bound on the permitted variance in the edge lengths in theorem \ref{thm:generic spanning trees} is suboptimal but requires minimal information on the structure of the graph.   
However, the constraint may be loosened in some situations, even if there is no additional information about the graph's structure.
For example, all eigenvalues of $\cL$ satisfy \cite{AndMor},
\begin{equation}\label{eq:better bound?}
\lambda_j \leq \mbox{max}_{(u,v)\in \E}(d_u+d_v)\ .
\end{equation}
As we assume in theorem \ref{thm:generic spanning trees} that we know the product of the vertex degrees, if we also know that $\mbox{max}_{(u,v)\in \E}(d_u+d_v)<V$, then we can use (\ref{eq:better bound?}) in place of theorem \ref{thm:upper bound on spectrum of L} which weakens the constraint on the spread of the edge lengths.

\backmatter

\bmhead{Acknowledgments}

The authors would like to thank Gregory Berkolaiko for helpful comments.
This work was partially supported by a grant from the Simons Foundation (354583 to Jonathan Harrison).

\bibliography{sn-bibliography}% common bib file

\begin{thebibliography}{99}

	\bibitem{AndMor}
	Anderson, W.~N. and Morley, T.~D.:
	\newblock Eigenvalues of the Laplacian of a graph,
	\newblock Linear Multilinear Algebra (1985). \href{https://doi.org/10.1080/03081088508817681}{https://doi.org/10.1080/03081088508817681}
	
	\bibitem{BPB09}
	Band, R., Parzanchevski, O., and Ben-Shach, G.:
	\newblock The isospectral fruits of representation theory: quantum graphs and drums,
	\newblock J. Phys. A (2009).
	\href{https://doi.org/10.1088/1751-8113/42/17/175202}{https://doi.org/10.1088/1751-8113/42/17/175202}
	
	\bibitem{Below85}
	von Below, J.:
	\newblock A characteristic equation associated to an eigenvalue problem on {$c^2$}-networks,
	%\newblock Linear Algebra Appl. 71:309-325 (1985)
	\newblock Linear Algebra Appl. (1985). \href{https://doi.org/10.1016/0024-3795(85)90258-7}{https://doi.org/10.1016/0024-3795(85)90258-7}
	
	\bibitem{BKbook}
	Berkolaiko, G. and Kuchment, P.:
	\newblock  Introduction to quantum graphs, vol. 186 of {\em Mathematical
		Surveys and Monographs}.
	\newblock Amer. Math. Soc., Providence, RI (2013)
	
	\bibitem{Bro99}
	Brooks, R.:
	\newblock Non-Sunada graphs,
	\newblock Ann. Inst. Fourier (1999). 
	\href{https://doi.org/10.5802/aif.1688}{https://doi.org/10.5802/aif.1688}
	
	\bibitem{Cha95}
	Chapman, J.~S.:
	\newblock Drums that sound the same,
	\newblock Amer. Math. Monthly (1995). %102:124-138 
	 \href{https://doi.org/10.2307/2975346}{https://doi.org/10.2307/2975346}
	
	\bibitem{Chungbook}
	Chung, F.:
	\newblock Spectral graph theory, vol. 92 of {\em CBMS Regional Conference Series in Mathematics}. 
	\newblock Amer. Math. Soc., Providence, RI (1997)
	
	\bibitem{ChYa}
	Chung, F. and Yau, S.~T.:
	\newblock Coverings, heat kernels and spanning trees.
	\newblock Electon. J. Combin.  6:R12 (1999)
	
	\bibitem{Des00}
	Desbois, J.:
	\newblock Spectral determinant of Schr\"odinger operators on graphs,
	\newblock J. Phys. A (2000).
	\href{https://doi.org/10.1088/0305-4470/33/7/103}{https://doi.org/10.1088/0305-4470/33/7/103}
	
	\bibitem{Des01}
	Desbois, J.:
	\newblock Spectral determinant on graphs with generalized boundary conditions,
	\newblock Eur. Phys. J. B (2001).
	\href{https://doi.org/10.1007/s100510170013}{https://doi.org/10.1007/s100510170013}
	
	\bibitem{Fri06}
	Friedlander, L.:
	\newblock Determinant of the Schr\"odinger operator on a metric graph,
	\newblock In: Berkolaiko, G., Carlson, R., Fulling, S., and Kuchment, P. (eds.) Quantum graphs and their applications, pp. 151-160. Amer. Math. Soc., Providence, RI (2006).  \href{https://doi.org/10.1090/conm/415/07866}{https:doi.org/10.1090/conm/415/07866}
	%Contemp. Math.
	
	\bibitem{GWW92}
	Gordon, C., Webb, D.~L., and Wolpert, S.:
	\newblock One cannot hear the shape of a drum,
	\newblock Bull. Am. Math. Soc. %27:134 
	(1992). \href{https://doi.org/10.1090/S0273-0979-1992-00289-6}{https://doi.org/10.1090/S0273-0979-1992-00289-6}
	
	
	\bibitem{GS01}
	Gutkin, B. and Smilansky, U.:
	\newblock Can one hear the shape of a graph?
	\newblock J. Phys. A (2001). \href{https://doi.org/10.1088/0305-4470/34/31/301}{https://doi.org/10.1088/0305-4470/34/31/301}
	
	
	\bibitem{HarKir11}
	Harrison, J.~M. and Kirsten, K.:
	\newblock Zeta functions of quantum graphs,
	\newblock  J. Phys. A  (2011). \href{https://doi.org/10.1088/1751-8113/44/23/235301}{https://doi.org/10.1088/1751-8113/44/23/235301}
	%44(23):235301, 29
	
		\bibitem{HarKirTex12}
	Harrison, J.~M., Kirsten, K., and Texier C.:
	\newblock Spectral determinants and zeta functions of Schr\"odinger operators on metric graphs,
	\newblock  J. Phys. A  (2012). \href{https://doi.org/10.1088/1751-8113/45/12/125206}{https://doi.org/10.1088/1751-8113/45/12/125206}
	%44(23):235301, 29
	
	\bibitem{HarWey18}
	Harrison, J.~M. and Weyand, T.:
	\newblock Relating zeta functions of discrete and quantum graphs,
	\newblock  Lett. Math. Phys.  (2018). \href{https://doi.org/10.1007/s11005-017-1017-0}{https://doi.org/10.1007/s11005-017-1017-0}
	%108(2) 377--390
	
	\bibitem{Kac66}
	Kac, M.:
	\newblock Can one hear the shape of a drum?,
	\newblock Am. Math. Monthly (1966). \href{https://doi.org/10.2307/2313748}{https://doi.org/10.2307/2313748}
	
	\bibitem{Kel67}
	Kelmans, A.~K.:
	\newblock On properties of the characteristic polynomial of a graph.
	\newblock Kibernetiku na slu\'zbu kommunizmu, 4, 27-41 (1967)
	
	\bibitem{Kirchhoff}
	Kirchhoff, G.:
	\newblock \"Uber die Aufl\"osung der Gleichungen, auf welche man bei der Untersuchung der linearen Verteilung galvanischer Str\"me gef\"uhrt wird.
	\newblock Ann. Phys. Chem. 72, 497-508 (1847) 
	
	\bibitem{Kuc03}
	Kuchment, P.:
	\newblock Quantum graphs {I}. {S}ome basic structures,
	\newblock Waves Random Media (2003). \href{https://doi.org/10.1088/0959-7174/14/1/014}{https://doi.org/10.1088/0959-7174/14/1/014}
	%14(1) S107--S128%\newblock Special section on quantum graphs.
	
	\bibitem{Mer94}
	Merris, R.:
	\newblock Laplacian matrices of graphs: A survey,
	\newblock Linear Algebra Appl. (1994). \href{https://doi.org/10.1016/0024-3795(94)90486-3}{https://doi.org/10.1016/0024-3795(94)90486-3}
	
	\bibitem{Moh91}
	Mohar, B.:
	\newblock Eigenvalues, diameter, and mean distance in graphs.
	\newblock Graphs Combin.  7, 53--64 (1991)
	
	\bibitem{NITS}
	Olver, F.~W.~J., Lozier, D.~W., Boisvert, R.~F., and  Clark, C.~W. (eds):
	\newblock  N{IST} handbook of mathematical functions,
	\newblock U.S. Department of Commerce, National Institute of Standards and
	Technology, Washington, DC. Cambridge University Press, Cambridge, U.K. (2010)
	
	\bibitem{Pank06}
	Pankrashkin, K.:
	\newblock Spectra of {S}chr\"odinger operators on equilateral quantum graphs,
	\newblock  Lett. Math. Phys.  (2006). \href{https://doi.org/10.1007/s11005-006-0088-0}{https://doi.org/10.1007/s11005-006-0088-0}
	%77(2) 139--154
	
	\bibitem{Pau36}
	Pauling, P.:
	\newblock The diamagnetic anisotropy of aromatic molecules,
	\newblock J. Chem. Phys.  (1936).  \href{https://doi.org/10.1063/1.1749766}{https://doi.org/10.1063/1.1749766}     
	
	\bibitem{Sun85}
	Sunada, T.:
	\newblock Riemannian coverings and isospectral manifolds,
	\newblock Ann. Math. (1985).
	\href{https://doi.org/10.2307/1971195}{https://doi.org/10.2307/1971195}
	
\end{thebibliography}

\end{document}